\documentclass[11pt]{amsart}

\usepackage{amsfonts,amssymb,amscd,amsmath,enumerate,verbatim,calc}
\usepackage{amsmath}
\usepackage{amssymb}
\usepackage{amscd}

\def\RMN#1{\uppercase\expandafter{\romannumeral#1}}

\newcommand{\proj}{\operatorname{Proj}}

\newcommand{\spec}{\operatorname{Spec}}

\newcommand{\rank}{\operatorname{rank}}

\newcommand{\Hom}{\operatorname{Hom}}

\newcommand{\G}{\operatorname{G_0}} %coherent sheaves

\newcommand{\subq}{_{\Bbb Q}}

\newtheorem{Theorem}{Theorem}[section]
\newtheorem{Lemma}[Theorem]{Lemma}
\newtheorem{Corollary}[Theorem]{Corollary}

\newtheorem{Definition}[Theorem]{Definition}
\newtheorem{Remark}[Theorem]{Remark}
\newtheorem{Example}[Theorem]{Example}

\newtheorem{thm}{Theorem}[section]

\newtheorem{lem}[thm]{Lemma}
\newtheorem{prop}[thm]{Proposition}

\newtheorem{cor}[thm]{Corollary}

\newtheorem{conj}[thm]{Conjecture}

\newcommand{\tensor}{\otimes}

\DeclareMathOperator{\Pic}{Pic}

 \DeclareMathOperator{\Tor}{Tor}
 \DeclareMathOperator{\Ext}{Ext}

 \DeclareMathOperator{\Se}{S}

 \DeclareMathOperator{\CH}{CH}

 \DeclareMathOperator{\pd}{pd}
 \DeclareMathOperator{\height}{height}

 \DeclareMathOperator{\depth}{depth}

\def\syz{\mathrm{\Omega}}

\DeclareMathOperator{\Cl}{Cl}

\DeclareMathOperator{\modu}{mod}
  \DeclareMathOperator{\MCM}{MCM} 
\newcommand{\ci}{\tau}
  \newcommand{\hc}{\theta_{ch}} 

\def\m{\mathfrak{m}}
\def\p{\mathfrak{p}}

\newcommand{\ZZ}{\mathbb{Z}}
\newcommand{\QQ}{\mathbb{Q}}

 \newcommand*{\SHom}{\mathcal{H}\mathit{om}}

\newcommand{\ses}[3]{0 \to {#1} \to {#2} \to {#3} \to 0}
 \newcommand{\h}[3]{\theta^{#1}({#2},{#3})}

\begin{document} 

\title{Hochster's theta pairing and numerical equivalence}

\author{Hailong Dao}
\address{Department of Mathematics\\
University of Kansas\\
 Lawrence, KS 66045-7523 USA}
\email{hdao@math.ku.edu}

\author{Kazuhiko Kurano}
\address{Department of Mathematics\\
School of Science and Technology\\
Meiji University\\
Higashimata 1-1-1, Tama-ku, Kawasaki-shi 214-8571, Japan}
\email{kurano@isc.meiji.ac.jp}

\date{}
\begin{abstract}
Let $(A,\m)$ be a local hypersurface with isolated singularity. We show that Hochster's theta pairing vanishes on elements that are {numerically equivalent to zero} in the Grothendieck group of $A$ under the mild assumption that $\spec A$ admits a resolution of singularity.  We also prove that when $\dim A =3$, the Hochster's
theta pairing is positive semidefinite. These results combine to show that the counter-example of Dutta-Hochster-McLaughlin to general vanishing of Serre's intersection multiplicity exists for any three dimensional isolated hypersurface singularity that is not a UFD and has a desingularization. Our method involves showing that $\theta^A$ gives a bivariant class for the morphism $\spec A/\m \to \spec A$. 
We also show that, if $A$ is three dimensional isolated hypersurface singularity that has a desingularization, the divisor class group
is finitely generated torsion-free.

\end{abstract}

\thanks{The first author is partially supported by NSF grants DMS 0834050 and DMS 1104017. The second author is supported by KAKENHI (21540050). }

%\abstract{We show that Hochster's theta pairing vanishes on elements {numerically equivalent to zero} in the Grothendieck group pf $A$.  }
\maketitle

\section{Introduction}\label{intro}

Let $A$ be a local hypersurface with isolated singularity (so $A_\p$ is regular for each  non-maximal $\p \in \spec A$).
For any pair of  finitely  generated $A$-modules $M$ and $N$, $\ell(\Tor_i^A(M,N))<\infty$ for $i>\dim A $, here $\ell(-)$ denotes length. 
The function $\theta^A(M,N)$ was introduced by Hochster (\cite{Ho1}) to be:
$$ \theta^A(M,N) = \ell(\Tor_{2e+2}^A(M,N)) -\ell(\Tor_{2e+1}^A(M,N)) $$
where $e$ is any integer such that $2e \geq \dim A$. The function $ \theta^A(M,N)$ is additive on short exact sequences and thus define a pairing on the Grothendieck group of finitely generated modules $\G(A)$ or the reduced group $\underline {G_0}(A):= G_0(A)/\langle[A]\rangle$.  

The theta pairing has attracted quite a bit of attention lately due to its recently discovered connections to a number of diverse areas and problems (\cite{BV, CW, Da2, Da3, MPSW, PV}). For a brief history of these recent developments and how they connect to our work we refer to Section \ref{history}. 

A main result of the present  paper is  the following:

\begin{thm}\label{kazu1}
Assume that $\spec A$ admits a resolution of singularities. Then $\theta^A$ vanishes on elements in the Grothendieck group $G_0(A)$ that are \textit{numerically equivalent to zero}. 
 \end{thm}
 
The  concept of numerical equivalence over local rings was introduced by the second author in parallel with intersection theory on projective varieties.  An element $[M]$ of the Grothendieck group $G_0(A)$ is  numerically 
equivalent to zero if for any module $N$ of finite length and finite projective dimension, 
$$\chi^A(M,N):= \sum_{i\gg 0} (-1)^i\ell(\Tor_i^A(M,N)) = 0$$ (in general this should be defined using the category of perfect complexes with finite length homologies but it makes no difference in our case, see Section \ref{prep}).   

Since numerical equivalence follows from algebraic equivalence  even in our setting (Remark~\ref{algeq}), our main result implies that of \cite{CW} and \cite{MPSW}.  
Several  corollaries follow. It gives  a new proof on the vanishing of $\theta^A$ when  $\dim A$ is even (Conjecture \ref{modconj}, part (1)) in the graded case (proved in \cite{MPSW}). 

The proof of \ref{kazu1} contains some ingredients which we believe are of independent interests. In \ref{}, we give a criterion for a map from the Grothendieck group $G_0(A)$ to $\ZZ$ to arise from a perfect complex with finite length homologies. In addition, a key technical result shows that $\theta^A$ gives bivariant class on the Grothendieck groups (as well as the Chow groups). Recall that a bivariant class for a morphism of scheme $f: X\to Y$ gives 
a homomorphism for the Grothendieck groups $G_0(Y') \to G_0(X')$  for any fibre square. Such map commutes with pushforward, pullback and intersection product (for a precise definition see Section \ref{biva}).

The second main result is that in dimension three, $\theta^A$ is positive semidefinite.  

\begin{thm}\label{dim3}
Let $A$ be a local hypersurface with isolated singularity of dimension $3$. Then $\theta^A(M,M)\geq 0$ for any module $M$. Furthermore, if $M$ is reflexive of rank one, then equality holds if and only if $M$ is free. 
\end{thm}

By combining Theorems \ref{kazu1} and  \ref{dim3} one obtains (see Corollary \ref{dhm}) a vast generalization of the famous original example by Dutta-Hochster-McLaughlin of non-vanishing Serre's intersection multiplicity, which was constructed for $A=k[[x,y,u,v]]/(xy-uv)$:

\begin{cor}\label{DHM}
Let $A$ be a local hypersurface of dimension three with isolated singularity. Assume that $\spec A$ admits a resolution of singularity. Then for any  torsion $A$-module $M$, the following are equivalent:
\begin{enumerate}
\item There exists a  module $N$  of finite length and finite projective dimension such that $\chi^A(M,N) \neq 0.$
\item The divisor class of $M$ is non-trivial (for example, if $M=A/I$ for a non-free reflexive ideal $I$).  
\end{enumerate}
\end{cor}

As an added bonus, it also follows that in the situation above, the class group of $A$ is always finitely generated torsion-free
(Corollary~\ref{last}). 

The paper is organized as follows. In Section \ref{prep} we recall basic definitions and notations. Section \ref{history} briefly recall the recent history of Hochster's theta invariant as well as a group of motivating conjectures. Section \ref{biva} gives the definition of a bivariant class for Grothendieck groups and a criterion for a map from $G_0(A)$ to $\ZZ$ to arise from the Euler characteristic of perfect complexes with finite length homologies. Section \ref{thetabiva} contains the proof that $\theta^A$ gives a bivariant class on Grothendieck groups. The proof of  Theorem  \ref{kazu1} is contained in Section \ref{mainresult}. Finally, Section  \ref{psd} contains the proofs of Theorem \ref{dim3} and Corollary \ref{DHM} as well as several corollaries and partial results in small dimensions. 

\section{Notations and preliminaries}\label{prep}

Unless otherwise noticed, all schemes in this paper are of finite type over some  regular base scheme $S$ and all morphisms are of finite type.
When dealing with Chow groups of schemes we always consider the {\em relative dimension}
as in Chapter~20 in Fulton~\cite{F}. 

For a scheme $X$, $G_0(X)$ denotes the Grothendieck group of coherent sheaves on $X$. For each $i\geq 0$, $A_i(X)$ denotes the Chow group of $i$-dimensional cycles in $X$ modulo rational equivalence.  
%Let $\mathfrak{Coh} (X)$ denote the category of coherent sheaves on $X$. 
Let $\Cl(X)$  denote the  class group of $X$. When $X=\spec A$ is affine we shall write $G_0(A)$, $A_i(A)$ and $\Cl(A)$.  Let $\underline{G_0}(A):=G_0(A)/\mathbb Z[A]$ be the reduced Grothendieck group and $\underline{G_0}(A)_{\mathbb Q}:=\underline{G_0}(A)\tensor_{\mathbb Z}\mathbb Q$ be the reduced Grothendieck group of $A$ with rational coefficients.

Let $A$ be a local noetherian ring. We recall the concept of numerical equivalence for elements in Chow or Grothendieck groups. 
%Let $\modu(A)$ and $\MCM(A)$ be the category of finitely generated  and finitely generated  maximal Cohen-Macaulay $A$-modules, respectively. 
If $A$ is a homomorphic image of a regular local ring, then we have an isomorphism of vector spaces:
$$ \ci: G_0(A)_{\mathbb Q} \to A_*(A)_{\mathbb Q}  $$
We denote by $\ci_i$ the image of $\ci$ in $A_i(A)$, these are defined
by localized Chern characters.
In the case of codimension one cycles one does not need to tensor with $\QQ$. Namely, there is a map (see \cite{Ch}): 
$$c_1: G_0(A) \to A_{d-1}(A)$$
The map $c_1$ satisfies
\begin{itemize}
\item
$c_1([A]) = 0$,
\item
$c_1([A/I]) = -c_1([I]) = [\spec A/I]$ for each reflexive ideal $I$,
\item
$c_1([M]) = 0$ if $\dim M \le d-2$.
\end{itemize}
When $A$ is normal we can, and will identify $A_{d-1}(A)$ with the class group $\Cl(A)$ of $A$. 
For an $A$-module $M$, we have
\[
\tau_{d-1}([M]) = c_1([M]) - \frac{\rank_AM}{2}K_A
\]
in $A_{d-1}(A)_{\Bbb Q}$, where $K_A$ is the canonical divisor of $A$,
that is, $K_A = c_1([\omega_A])$.

For a bounded finite $A$-free complex ${\Bbb F}.$ 
with homology finite length,
we define
\[
\chi_{{\Bbb F}.} : G_0(A) \longrightarrow {\Bbb Z}
\]
to be 
\[
\chi_{{\Bbb F}.}([M]) = \sum_{i}(-1)^i\ell(H_i({\Bbb F}.\otimes_AM)) .
\]
We say that a cycle $\alpha$ in $G_0(A)$ is {\em numerically equivalent to
$0$} if $\chi_{{\Bbb F}.}(\alpha) = 0$ for any bounded finite $A$-free complex ${\Bbb F}.$ with homology finite length.
In the same way, we say that a cycle $\beta$ in $A_*(A)$
is {\em numerically equivalent to
$0$} if ${\rm ch}({\Bbb F}.)(\beta) = 0$ for any above ${\Bbb F}.$,
where ${\rm ch}({\Bbb F}.)$ is the localized Chern character
(see 17 in \cite{F}).
We denote by $\overline{G_0(A)}$ and $\overline{A_*(A)}$
the groups modulo numerical equovalence.
By Theorem~3.1 and Remark~3.5 in \cite{K23},
both of $\overline{G_0(A)}$ and $\overline{A_*(A)}$
are finitely generated torsion-free abelian group
under a mild condition.
It is proved in Proposition~2.4 in \cite{K23} that
numerical equivalence is consistent with dimension of cycles
in $A_*(A)$, so we have
\[
\overline{A_*(A)} = \oplus_{i = 0}^d\overline{A_i(A)} .
\]
The Riemann-Roch map $\tau$ preserves numerical equivalence
as in \cite{K23},
that is, it induces the map $\overline{\tau}$
that makes the following diagram commutative:
\[
\begin{array}{ccc}
G_0(A)_{\Bbb Q} & \stackrel{\tau}{\longrightarrow} & A_*(A)_{\Bbb Q} \\
\downarrow & & \downarrow \\
\overline{G_0(A)_{\Bbb Q}} & \stackrel{\overline{\tau}}{\longrightarrow} & \overline{A_*(A)_{\Bbb Q}}
\end{array}
\]
If $A$ is Cohen-Macaulay,
the Grothendieck group of bounded $A$-free complexes
with support in $\{ m \}$ is generated by finite free resolutions of modules of finite length and finite projective dimension (see Proposition~2 in \cite{RS}).
Therefore, in this case,
$\alpha$ in $G_0(A)$ is numerically equivalent to $0$
if and only if $\chi_{{\Bbb F}.}(\alpha) = 0$ for any 
free resolution of a module with finite length and
finite projective dimension.

\section{A brief history of theta invariant and some motivating open questions}\label{history}

In this section we briefly recall the (recent) history of the theta functions. Let $A$ be a local hypersurface with isolated singularity (so $A_\p$ is regular for each  non-maximal $\p \in \spec A$).
For any pair of  finitely  generated $A$-modules $M$ and $N$, $\ell(\Tor_i^A(M,N))<\infty$ for $i>\dim A $, here $\ell(-)$ denotes length. 
The function $\theta^A(M,N)$ was introduced by Hochster (\cite{Ho1}) to be:
$$ \theta^A(M,N) = \ell(\Tor_{2e+2}^A(M,N)) -\ell(\Tor_{2e+1}^A(M,N)) $$
where $e$ is any integer such that $2e \geq \dim A$. It is well known (see \cite{Ei}) that the sequence of modules $\{\Tor_i^A(M,N)\}$ is periodic of
period 2 for $i>\depth A- \depth M$, so this function is well-defined. Note that if $M$ or $N$ is maximal Cohen-Macaulay, one simply gets:
$$ \theta^A(M,N) = \ell(\Tor_{2}^A(M,N)) -\ell(\Tor_{1}^A(M,N)) $$

A key point here is the function $ \theta^A(M,N)$ is additive on short exact sequences and thus define a pairing on the Grothendieck group of finitely generated modules $G_0(A)$ or the reduced group $\underline {G_0}(A)$.  
This function was originally introduced as a possible mean to attack the Direct Summand Conjecture (where it was necessary to define $\theta^{A}$  for some non-isolated singularity $A$), however it has taken its own life recently. In particular, the behavior of $\theta^A$ over isolated singularities has been linked to some interesting topics in algebraic geometry and K-theory.

To be more specific, the main motivation of this note is  the following group of open questions:

\begin{conj}\label{modconj}

Let $A$ be a local hypersurface with isolated singularity and $d = \dim A$. Then for any finitely generated modules $M, N$: 
\begin{enumerate}
\item\label{con1} If $d$ is even, then $\h{A}{M}{N}=0$.
\item\label{con2} If $\dim M + \dim N \leq d$, then $\h{A}{M}{N}=0$.
\item\label{con3} If $\dim M  \leq d/2$, then $\h{A}{M}{N}=0$.
\item\label{con4} If  $M,N$ are maximal Cohen-Macaulay, then $\h{A}{M}{N} + (-1)^{\frac{d+1}{2}}\h{A}{M^*}{N}=0.$ Note that when $d$ is even this implies $(1)$.
\item\label{con5} (\cite[Conjecture 3.6]{MPSW})If $d$ is odd, then $(-1)^{\frac{d+1}{2}}\h{A}{M}{M}\geq 0$. In other words, $(-1)^{\frac{d+1}{2}}\h{A}{-}{-}$  defines a positive semi-definite form  on  $\underline {G_0}(A)_{\mathbb Q}$. 
\end{enumerate}
\end{conj}

Most of the statements above have appeared or hinted at in the literature in one form or another.  
Conjecture \ref{modconj} (\ref{con1}) was made and established in several cases  by the first author in \cite{Da2, Da1}. Since then, it has captured attention of many researchers  and has now been established in characteristic $0$ via three different approaches: intersection theory on smooth hypersurfaces (for the graded case), topological K-theory and Hochschild cohomology (\cite{MPSW}, \cite{BV}, \cite{PV}). These results suggest   much deeper facts about $\theta^A$, namely that it should be thought of as a Riemann-Roch form on the category of maximal Cohen-Macauley modules  (or matrix factorizations) over $A$. Thus they motivate the rest of Conjecture  \ref{modconj}, which can be viewed as analogues of Lefschetz hyperplane theorem and properties of Hodge-Riemann bilinear relations.  We note that (2) is proved for the excellent and equicharateristic case in \cite{Da1}. The statement (4) can be viewed as a strengthening of a result by Buchweitz which implies that $\h{A}{M}{N} + (-1)^{d+1}\h{A}{M^*}{N^*}=0$ in any dimension, see \cite[Proposition 4.3]{Da2}. 

In view of the above and the main results of this work, it is reasonable to make the following, which would explain some parts of Conjecture \ref{modconj}.
Precisely speaking, by Theorem~\ref{kazu1},
we know that
Conjecture~\ref{modconj} (1) follows from 
Conjecture~\ref{conj} (1) (a).
Also Conjecture~\ref{modconj} (3) follows from 
Conjecture~\ref{conj} (1) (b).

\begin{conj}\label{conj}
Let $(A, m)$ be a $d$-dimensional local domain with isolated singularity.
\begin{enumerate}
\item
Assume that $A$ is a complete intersection.
\begin{enumerate}
\item
If $d$ is even, then $\overline{G_0(A)}\subq = {\Bbb Q}[A]$ and, 
equivalently, $\overline{A_i(A)}\subq = 0$ for $i < d$.
\item
If $d$ is odd, then $\overline{A_i(A)}\subq = 0$ for $i \neq  \frac{d+1}{2}, d$.
\end{enumerate}
\item
If $i \le d/2$, then $\overline{A_i(A)}\subq = 0$.
\end{enumerate}
\end{conj}

This conjecture is true in some special cases as follows.

\begin{prop}\label{affirm}
Assume that $R$ is a homogeneous ring over a field $k$ such that $\proj R$ is smooth over $k$.
Put $A = R_{R_+}$, where $R_+$ is the unique homogeneous maximal ideal.

Then, Conjecture~\ref{conj} (1) is true.
Furthermore, if the Grothendieck's  standard conjecture is true, then 
Conjecture~\ref{conj} (2) is true.
\end{prop}

\proof
Put $X = \proj R$.
Suppose $n = \dim X$ and $d = \dim A$.
Note that $d = n+1$. 

We may assume that $k$ is algebraically closed.

Suppose that the  characteristic of $k$ is $0$.
We may assume that $k$ is the complex number field ${\Bbb C}$ by the
Lefschetz principle.

Assume that $X$ is a complete intersection smooth projective variety over ${\Bbb C}$.
Then, it is well known that 
\[
H^j(X({\Bbb C}), {\Bbb Q})
=\left\{
\begin{array}{cl}
{\Bbb Q} & (\mbox{$0 \le j \le 2n$, $j \neq n$ and $j$ is even}) \\
? & (j = n) \\
0 & (\mbox{otherwise}) .
\end{array}
\right.
\]
Then, we have 
\[
\begin{array}{ccccc}
{\rm CH}^j(X)_{\Bbb Q} & \longrightarrow & 
{\rm CH}_{hom}^j(X)_{\Bbb Q} & \longrightarrow & 
{\rm CH}_{num}^j(X)_{\Bbb Q} \\
& \searrow & \downarrow & & \\
& & H^{2j}(X({\Bbb C}), {\Bbb Q}) & & 
\end{array}
\]
where ${\rm CH}_{hom}^j(X)_{\Bbb Q}$ 
(resp.\ ${\rm CH}_{num}^j(X)_{\Bbb Q}$) is the Chow group
divided by homological equivalence (resp.\ numerical equivalence).
Here, the map ${\rm CH}_{hom}^j(X)_{\Bbb Q} \rightarrow 
H^{2j}(X({\Bbb C}), {\Bbb Q})$ is injective, and 
${\rm CH}_{hom}^j(X)_{\Bbb Q} \rightarrow 
{\rm CH}_{num}^j(X)_{\Bbb Q}$ is surjective.
Remember that ${\rm CH}_{num}^j(X)_{\Bbb Q} \neq 0$ for $j = 0, 1, \ldots, n$.
Therefore, if $n$ is odd,
${\rm CH}_{num}^j(X)_{\Bbb Q} = {\Bbb Q}$ for $j = 0, 1, \ldots, n$.
If $n$ is even,
${\rm CH}_{num}^j(X)_{\Bbb Q} = {\Bbb Q}$ for $j = 0, 1, \ldots, n$
except for $j = n/2$.
On the other hand, we have the natural surjections
\[
{\rm CH}_{num}^j(X)_{\Bbb Q}/h{\rm CH}_{num}^{j-1}(X)_{\Bbb Q} 
\longrightarrow
\overline{A_{d-j}(A)}_{\Bbb Q}
\]
for $j = 0, 1, \ldots, n$ by (7.5) in \cite{K23},
where $h$ is the very ample divisor corresponding to the embedding $\proj R$.
Remark that $\overline{A_0(A)}_{\Bbb Q} = 0$ if $\dim A > 0$.
\begin{itemize}
\item
If $n$ is odd,
then $\overline{A_i(A)}\subq = 0$ for $i < d$.
Conjecture~\ref{conj} (1) (a) is proved.
\item
If $n$ is even,
then $\overline{A_i(A)}\subq = 0$ for $i \neq (d+1)/2, d$.
Conjecture~\ref{conj} (1) (b) is proved.
\end{itemize}
Conjecture~\ref{conj} (2) is true by Remark~7.12  in \cite{K23}.

In the case where the characteristic of $k$  is positive,
we use the \'etale cohomology instead of the Betti cohomology.
The proof is the same as the case of characteristic zero, so
we omit it.
\qed

The localized Chern characters  induce a map $\hc^A: A_*(A)_{\mathbb Q} \times A_*(A)_{\mathbb Q}  \to {\mathbb Q} $. Thus, we can also state the following conjecture for Chow groups.  

\begin{conj}\label{}
For $\alpha \in A_*(A)$ a homogenous element, we have $\hc^A(\alpha, \beta) =  0$ if $\alpha \notin A_{\frac{d+1}{2}} (A)$(in particular, it is always $0$ when $d$ is even). Also,  $(-1)^{\frac{d+1}{2}}\hc^A(-,-)$ defines a positive semi-definite form on $A_{\frac{d+1}{2}}(A)$.
\end{conj}

\section{A bivariant class on K-groups}\label{biva}

\begin{Definition}\label{bivariantK}
\begin{rm}
Let $Y$ be a scheme and $X$ a closed subscheme of $Y$.
Consider the following fibre square of schemes:
\begin{equation}\label{eq1}
\begin{array}{ccc}
X' & \longrightarrow & Y' \\
\downarrow & \Box & \phantom{\scriptstyle g}\downarrow {\scriptstyle g}\\
X & \longrightarrow & Y
\end{array} 
\end{equation}
Suppose that 
\[
\varphi_g : G_0(Y') \longrightarrow G_0(X') 
\]
is a homomorphism between K-groups,
where $G_0( \ )$ denotes the Grothendieck group of
coherent sheaves.

We say that a collection of homomorphisms
\[
\{ 
\varphi_g : G_0(Y') \longrightarrow G_0(X') \mid
\mbox{$g : Y' \rightarrow Y$ is a morphism of schemes} 
\}
\]
is a {\em bivariant class} on K-groups for $X \hookrightarrow Y$
if the following three conditions are satisfied.
\begin{itemize}
\item[(${\rm B}_1$)]
If $h : Y" \rightarrow Y'$ is proper, $g : Y' \rightarrow Y$ arbitrary, and one forms the fibre diagram
\begin{equation}\label{eq2}
\begin{array}{ccc}
X" & \longrightarrow & Y" \\
{\scriptstyle h'}\downarrow \phantom{\scriptstyle h'} & \Box & \phantom{\scriptstyle h}\downarrow {\scriptstyle h} \\
X' & \longrightarrow & Y' \\
\downarrow & \Box & \phantom{\scriptstyle g}\downarrow {\scriptstyle g} \\
X & \longrightarrow & Y
\end{array} ,
\end{equation}
then the diagram
\[
\begin{array}{ccc}
G_0(Y") & \stackrel{\varphi_{gh}}{\longrightarrow} & G_0(X") \\
{\scriptstyle h_*}\downarrow \phantom{\scriptstyle h_*} &  & 
\phantom{\scriptstyle h'_*}\downarrow {\scriptstyle h'_*} \\
G_0(Y') & \stackrel{\varphi_{g}}{\longrightarrow} & G_0(X') 
\end{array}
\]
is commutative,
where $h_*$ is the push-forward map induced by the proper morphism
$h$ (e.g.\ \cite{F}).
\item[(${\rm B}_2$)]
If $h : Y" \rightarrow Y'$ is flat, $g : Y' \rightarrow Y$ arbitrary, and one forms the fibre diagram~(\ref{eq2}),
then the diagram
\[
\begin{array}{ccc}
G_0(Y') & \stackrel{\varphi_{g}}{\longrightarrow} & G_0(X') \\
{\scriptstyle h^*}\downarrow \phantom{\scriptstyle h^*} &  & 
\phantom{\scriptstyle {h'}^*}\downarrow {\scriptstyle {h'}^*} \\
G_0(Y") & \stackrel{\varphi_{gh}}{\longrightarrow} & G_0(X") 
\end{array}
\]
is commutative,
where $h^*$ is the pull-back map induced by the flat morphism
$h$ (e.g.\ \cite{F}).
\item[(${\rm B}_3$)]
Let $Z'$ be a scheme and $Z"$ be a closed subscheme of $Z'$.
Let ${\Bbb G}.$ be a bounded locally free complex on $Z'$ which is exact on $Z' \setminus Z"$.
If $g : Y' \rightarrow Y$, $h : Y' \rightarrow Z'$ are morphisms, and one forms the fibre diagram
\begin{equation}\label{B3}
\begin{array}{ccccc}
X" & \longrightarrow & Y" & \longrightarrow & Z" \\
{\scriptstyle i"}\downarrow \phantom{\scriptstyle i"} & \Box & \phantom{\scriptstyle i'}\downarrow {\scriptstyle i'} & \Box & \phantom{\scriptstyle i}\downarrow {\scriptstyle i}\\
X' & \stackrel{f}{\longrightarrow} & Y'  & \stackrel{h}{\longrightarrow} & Z' \\
\downarrow & \Box & \phantom{\scriptstyle g}\downarrow {\scriptstyle g} & & \\
X & \longrightarrow & Y & & 
\end{array} ,
\end{equation}
then the diagram
\[
\begin{array}{ccc}
G_0(Y') & \stackrel{\varphi_{g}}{\longrightarrow} & G_0(X') \\
{\scriptstyle \chi_{h^*({\Bbb G}.)}}\downarrow \phantom{\scriptstyle \chi_{h^*({\Bbb G}.)}}
&  & \phantom{\scriptstyle \chi_{(hf)^*({\Bbb G}.)}}\downarrow {\scriptstyle \chi_{(hf)^*({\Bbb G}.)}} \\
G_0(Y") & \stackrel{\varphi_{gi'}}{\longrightarrow} & G_0(X") 
\end{array}
\]
is commutative,
where $\chi_{h^*({\Bbb G}.)}$ is the map taking the
alternating sum of the homologies of the complex 
$h^*({\Bbb G}.) \otimes_{{\mathcal O}_{Y'}} {\mathcal F}$
for a coherent ${\mathcal O}_{Y'}$-module ${\mathcal F}$.
\end{itemize} 
\end{rm}
\end{Definition}

Under a mild condition, we can naturally define a bivariant class 
on Chow groups in the sense of
Fulton (\cite{F}, Section~17)
corresponding to a bivariant class on K-groups in Definition~\ref{bivariantK}
(see Remark~\ref{bivariantC}).

\begin{Example}
\begin{rm}
Let $X$ be a closed subscheme of a scheme $Y$.
Let ${\Bbb F}.$ be a bounded locally free ${\mathcal O}_X$-complex
which is exact on $Y \setminus X$.
For a fibre square as in (\ref{eq1}),
we define
\[
\chi({\Bbb F}.)_g : G_0(Y') \longrightarrow G_0(X') 
\]
by 
\[
\chi({\Bbb F}.)_g([{\mathcal F}])
= \sum_{i}(-1)^i[H_i(g^*({\Bbb F}.) \otimes_{{\mathcal O}_{Y'}} {\mathcal F})] .
\]
Then, 
\[
\{ 
\chi({\Bbb F}.)_g : G_0(Y') \longrightarrow G_0(X') \mid
\mbox{$g : Y' \rightarrow Y$ is a morphism of schemes} 
\}
\]
is a bivariant class on K-groups for $X \hookrightarrow Y$.

It is easy to check that (${\rm B}_2$) and (${\rm B}_3$) are satisfied.
The condition (${\rm B}_1$) is proved using a spectral sequence.
For a coherent ${\mathcal O}_{Y"}$-module ${\mathcal F}$,
we show that both of $h'_* \chi({\Bbb F}.)_{gh}([{\mathcal F}])$ 
and $\chi({\Bbb F}.)_{g}h_*([{\mathcal F}])$ coincides with
\[
\sum_i(-1)^i[Rh_*^i((gh)^*({\Bbb F}.)\otimes_{{\mathcal O}_{Y"}}{\mathcal F})] \in G_0(X') .
\]
\end{rm}
\end{Example}

\begin{Remark}\label{bivariantC}
\begin{rm}
Let $Y$ be a scheme and $X$ a closed subscheme of $Y$.
Assume that there exists a proper surjective morphism $Z \rightarrow Y$
such that $Z$ is a regular scheme.

If we consider Grothendieck groups and Chow groups with
rational coefficients,
there exists the natural one-to-one corresponding 
between bivariant classes on K-groups for $X \hookrightarrow Y$ and 
bivariant classes on Chow groups  for $X \hookrightarrow Y$ as follows.

\begin{enumerate}
\item 
Let
\[
\{ 
\varphi_g : G_0(Y')\subq \longrightarrow G_0(X')\subq \mid
\mbox{$g : Y' \rightarrow Y$ is a morphism of schemes} 
\}
\]
be a bivariant class on K-groups for $X \hookrightarrow Y$
as in Definition~\ref{bivariantK}.

We denote the composite map of
\[
A_*(Y')\subq \stackrel{\tau_{Y'}^{-1}}{\longrightarrow} G_0(Y')\subq \stackrel{\varphi_g}{\longrightarrow} G_0(X')\subq \stackrel{\tau_{X'}}{\longrightarrow} A_*(X')\subq
\]
by $c_g : A_*(Y')\subq \rightarrow A_*(X')\subq$,
where $\tau_{Y'}$ and $\tau_{X'}$ are the isomorphisms given by 
Riemann-Roch theorem (Chapter~18 and 20 in Fulton~\cite{F})
in the category of $S$-schemes.

Then, the collection of homomorphisms
\[
\{  c_g : A_*(Y')\subq \rightarrow A_*(X')\subq \mid
\mbox{$g : Y' \rightarrow Y$ is a morphism of schemes} 
\}
\]
satisfies the conditions  (${\rm C}_1$),  (${\rm C}_2$) and  (${\rm C}_3$)
in Definition~17.1 in Fulton~\cite{F}, i.e., it is a bivariant class
on Chow groups for $X \hookrightarrow Y$.

The condition (${\rm C}_1$) is easily checked by definition.
The condition (${\rm C}_3$) is proved using 
Theorem~17.1 in Fulton~\cite{F}.
It is a troublesome task to prove (${\rm C}_2$),
since the Riemann-Roch map $\tau$ can not commute with flat pull-back maps.
In order to prove (${\rm C}_2$), we need the assumption that 
there exists a proper surjective morphism $\pi : Z \rightarrow Y$
such that $Z$ is a regular scheme.
We just give a sketch of a proof here.
(We do not use this result in this paper.)
First, we show that, there exists a bounded locally free 
${\mathcal O}_Z$-complex ${\Bbb F}.$ such that,
for any $Z$-scheme $h : Z' \rightarrow Z$,
$\varphi_{\pi h}$ coincides with $\chi({\Bbb F}.)_h$.
Using the fact that localized Chern characters 
are bivariant classes on Chow group (in particular,
compatible with flat pull-back maps), we can prove (${\rm C}_2$).

\item
Conversely, let
\[
\{  c_g : A_*(Y')\subq \rightarrow A_*(X')\subq \mid
\mbox{$g : Y' \rightarrow Y$ is a morphism of schemes} 
\}
\]
be a bivariant class on Chow groups for $X \hookrightarrow Y$.

We denote the composite map of
\[
G_0(Y')\subq \stackrel{\tau_{Y'}}{\longrightarrow} A_*(Y')\subq \stackrel{c_g}{\longrightarrow} A_*(X')\subq \stackrel{\tau_{X'}^{-1}}{\longrightarrow} G_0(X')\subq
\]
by $\varphi_g : G_0(Y')\subq \rightarrow G_0(X')\subq$.

Then, the collection of homomorphisms
\[
\{ 
\varphi_g : G_0(Y')\subq \longrightarrow G_0(X')\subq \mid
\mbox{$g : Y' \rightarrow Y$ is a morphism of schemes} 
\}
\]
is a bivariant class on K-groups for $X \hookrightarrow Y$.

The condition (${\rm B}_1$) is easily checked by definition.
The condition (${\rm B}_3$) is proved using 
the fact that a localized Chern character is
commutative with any bivariant class on Chow groups
as in Roberts~\cite{Ro}.
It is a delicate task to prove (${\rm B}_2$),
since the Riemann-Roch map $\tau$ can not commute with flat pull-back maps.
In order to prove (${\rm B}_2$), we need the assumption that 
there exists a proper surjective morphism $Z \rightarrow Y$
such that $Z$ is a regular scheme. Then
one can prove (${\rm B}_2$) in the same way as (${\rm C}_2$)
in (1).
\end{enumerate}
\end{rm}
\end{Remark}

In the rest of this section, we give a sufficient condition
for a bivariant class of K-groups to coincide with
$\{ \chi({\Bbb F}.)_g \}$ for some bounded finite $A$-free complex ${\Bbb F}.$.

\begin{Theorem}\label{maintheorem1}
Let $A$ be a Noetherian local domain that is a homomorphic image 
of a regular local ring.
Assume that $\dim A > 0$.
Let $I$ be a non-zero ideal of $A$.
Put $Y = \spec A$ and $X = \spec A/I$.

We assume that there exists a resolution of singularity of $Y$,
i.e., 
a proper birational morphism $\pi : Z \rightarrow Y$
such that $Z$ is regular.
Put $W = \pi^{-1}(I)$ and $U = Z \setminus W$.
Assume that $U$ is isomorphic to 
$Y \setminus X$.
Let $i: W \rightarrow Z$ be the inclusion.
Let $i_* : G_0(W) \rightarrow G_0(Z)$  be
the induced map by $i$.

Let
\[
\{ 
\varphi_g : G_0(Y') \longrightarrow G_0(X') \mid
\mbox{$g : Y' \rightarrow Y$ is a morphism of schemes} 
\}
\]
be a bivariant class on K-groups for $X \hookrightarrow Y$.

If $i_*\varphi_\pi([{\mathcal O}_Z]) = 0$ in $G_0(Z)$,
then there exists a bounded $A$-free complex ${\Bbb F}.$
satisfing the following conditions:
\begin{enumerate}
\item
The complex ${\Bbb F}.$ is exact on $Y \setminus X$.
\item
For any morphism of schemes $g : Y' \rightarrow Y$,
the map 
\[
\varphi_g \otimes 1 : G_0(Y')\subq \longrightarrow G_0(X')\subq
\]
 coincides with
\[
\chi({\Bbb F}.)_g \otimes 1 : 
G_0(Y')\subq \longrightarrow G_0(X')\subq .
\]
\end{enumerate}
\end{Theorem}

Remark that, if $A$ is an excellent local domain containing 
${\Bbb Q}$, and regular on $\spec A \setminus \spec A/I$,
there exists a resolution of singularity of $\spec A$
satisfying the condition in Theorem~\ref{maintheorem1}.
For any excellent local domain of any characteristic,
it is expected to exist such a resolution of singularities.

\proof
By Thomason-Trobaugh~\cite{TT}, we have the following commutative diagram
\[
\begin{array}{ccccc}
& & G_0(W) & \stackrel{i_*}{\longrightarrow} & G_0(Z) \\
& & \phantom{\scriptstyle \eta}\|{\scriptstyle \eta} & & \| \\
K_1(U) & \longrightarrow & K_0^W(Z) & \longrightarrow & K_0(Z) \\
\| & & \phantom{\scriptstyle \pi^*}\uparrow{\scriptstyle \pi^*} & & \uparrow  \\
K_1(U) & \longrightarrow & K_0^I(A) & \longrightarrow & K_0(A)
\end{array} ,
\]
where $K_0^I(A)$ is the Grothendieck group of perfect $A$-complexes
which are exact outside of $X = \spec A/I$.
The map $\eta$ takes the alternating sum of homologies of
complexes.

Since $\dim A > 0$, the homomorphism $K_0^I(A) \rightarrow K_0(A)$ is zero.

Therefore, we have an exact sequence
\begin{equation}\label{exact}
K_0^I(A)  \stackrel{\eta\pi^*}{\longrightarrow}
G_0(W) \stackrel{i_*}{\longrightarrow} G_0(Z) .
\end{equation}

Consider the cycle $\varphi_\pi ([{\mathcal O}_Z]) \in G_0(W)$.
By the assumption, we have 
\[
i_*\varphi_\pi ([{\mathcal O}_Z]) = 0
\]
in $G_0(Z)$.
By the exact sequence (\ref{exact}),
we have a bounded $A$-free complex ${\Bbb F}.$ 
which is exact on $Y \setminus X$ such that
\begin{equation}\label{siki}
\varphi_\pi ([{\mathcal O}_Z]) = 
\eta\pi^*([{\Bbb F}.]) = 
\chi({\Bbb F}.)_\pi ([{\mathcal O}_Z]) 
\end{equation}
in $G_0(W)$.
Here, remark that, for a complex ${\Bbb G}.$,
$-[{\Bbb G}.]$ coincides with the shifted complex $[{\Bbb G}.(-1)]$
in the Grothendieck group of complexes since the mapping cone 
$C({\Bbb G}. \stackrel{1}{\rightarrow} {\Bbb G}.)$ is exact, and
the sequence of complexes
\[
0 \longrightarrow {\Bbb G}. \longrightarrow 
C({\Bbb G}. \stackrel{1}{\rightarrow} {\Bbb G}.)
\longrightarrow {\Bbb G}.(-1) \longrightarrow 0
\]
is exact.
Therefore, we can choose a complex ${\Bbb F}.$
satisfying (\ref{siki}).

We shall show that the map
$\varphi_g \otimes 1$
 coincides with
$\chi({\Bbb F}.)_g \otimes 1$ for any $g$.

Remark that 
\[
\{ 
\varphi_g - \chi({\Bbb F}.)_g : G_0(Y') \rightarrow G_0(X') \mid
\mbox{$g : Y' \rightarrow Y$ is a morphism of schemes}
\}
\]
is also a bivariant class on K-groups for $X \hookrightarrow  Y$.
By (\ref{siki}), it is satisfied that
\[
(\varphi_{\pi} - \chi({\Bbb F}.)_{\pi}) ([{\mathcal O}_Z]) = 0
\]
in $G_0(W)$.

We have only to prove the following lemma.

\begin{Lemma}
Let $Y$ be a scheme and $X$ be a closed subscheme of $Y$.
Let $\pi : Z \rightarrow Y$ be a proper surjective morphism such that $Z$ is a regular scheme.
Let
\[
\{ 
\phi_g : G_0(Y') \rightarrow G_0(X') \mid
\mbox{$g : Y' \rightarrow Y$ is a morphism of schemes} 
\}
\]
be a bivariant class on K-groups for $X \hookrightarrow  Y$.
Let $W = \pi^{-1}(X)$.

If $\phi_\pi([{\mathcal O}_Z]) = 0$ in $G_0(W)$,
then
\[
\phi_g \otimes 1 : G_0(Y')\subq \longrightarrow G_0(X')\subq
\]
is zero for any morphism $g : Y' \rightarrow Y$.
\end{Lemma}

\proof
We shall prove it in two steps.

%\noindent
%Step~1.
%Let $g : Y' \rightarrow Y$ be a morphism such that $Y'$ is regular.
%We shall prove that if $\phi_g([{\mathcal O}_{Y'}]) = 0$, 
%then $\phi_g = 0$.
%
%Let ${\mathcal F}$ be a coherent ${\mathcal O}_{Y'}$-module.
%We want to show $\phi_g([{\mathcal F}]) = 0$.
%
%Let ${\Bbb G}.$ be a finite locally free  ${\mathcal O}_{Y'}$-resolution
%of ${\mathcal F}$.
%Then, by (${\rm B}_3$),
%we have
%\[
%\phi_g([{\mathcal F}]) = \phi_g\chi_{{\Bbb G}.}([{\mathcal O}_{Y'}])
%= \chi_{{\Bbb G}.\otimes {\mathcal O}_{X'}}\phi_g([{\mathcal O}_{Y'}]) = 0 .
%\]
%
%\vspace{2mm}
%
Step~1.
Let $f : Z' \rightarrow Z$ be a morphism such that $Z'$ is an integral scheme.
We shall prove that $\phi_{\pi f}([{\mathcal O}_{Z'}]) = 0$.

We shall prove it by induction on $\dim Z'$. 
Suppose that the assertion is true if the dimension is less than $\dim Z'$.
By Nagata's compactification and (${\rm B}_2$),
we may assume that $f$ is a proper morphism.
By Chow's lemma,
there exists a proper birational morphism $h : Z" \rightarrow Z'$ such that
$f h : Z" \rightarrow Z$ is projective.
By (${\rm B}_1$) and induction on dimension,
it is enough to show that $\phi_{\pi f h}([{\mathcal O}_{Z"}]) = 0$.
So, we may assume that $f$ is projective.
Then $f$ can be factored as 
\[
Z' \stackrel{i}{\hookrightarrow} V \stackrel{p}{\longrightarrow} Z
\]
where $p : V \rightarrow Z$ is a smooth morphism  and 
$i$ is a closed immersion.
By (${\rm B}_2$),  $\phi_{\pi p}([{\mathcal O}_V]) = 0$.
Let ${\Bbb G}.$ be a bounded locally free  ${\mathcal O}_V$-resolution
of ${\mathcal O}_{Z'}$.
Then, by (${\rm B}_3$),
we have
\[
\phi_{\pi f}([{\mathcal O}_{Z'}]) = \phi_{\pi f}\chi_{{\Bbb G}.}([{\mathcal O}_{V}])
= \chi_{{\Bbb G}.\otimes {\mathcal O}_{V'}}\phi_{\pi p}([{\mathcal O}_V]) = 0 ,
\]
where $V' = V \times_Y X$.

\vspace{2mm}

Step~2.
Let $g : Y' \rightarrow Y$ be a morphism such that $Y'$ is an integral scheme.
We shall prove that $\phi_{g}([{\mathcal O}_{Y'}])$ is a torsion of $G_0(X')$ by induction on $\dim Y'$. 
Suppose that the assertion is true if the dimension is less than $\dim Y'$.
Consider the following fibre square:
\[
\begin{array}{ccc}
Z' & \stackrel{g'}{\longrightarrow} & Z \\
{\scriptstyle \pi'}\downarrow \phantom{\scriptstyle \pi'} & \Box & \phantom{\scriptstyle \pi}\downarrow {\scriptstyle \pi} \\
Y' & \stackrel{g}{\longrightarrow} & Y
\end{array} 
\]
Let $i : Z" \rightarrow Z'$ be a closed immersion such that
$Z"$ is integral and  $\pi' i : Z" \rightarrow Y'$ is
generically finite and surjective.
Then, by Step~1 and (${\rm B}_1$), 
$\phi_{g}((\pi' i)_*([{\mathcal O}_{Z"}])) = 0$.
We can show that $\phi_{g}([{\mathcal O}_{Y'}])$ is a torsion
using induction on dimension.
\qed

\section{Theta invariant as a bivariant class on K-groups}\label{thetabiva}

\begin{Definition}\label{theta}
\begin{rm}
Let $(A, m)$ be a hypersurface.
Let $N$ be a finitely generated $A$-module, 
and ${\Bbb F}.$ be an $A$-free resolution of $N$.
Let $I$ be an ideal of $A$.
Assume that $N_P$ is an $A_P$-module of finite projective dimension
for $P \in \spec A \setminus V(I)$.
Put $Y = \spec A$ and $X = \spec A/I$.

For a fibre square~(\ref{eq1}),
consider the homomorphism
\[
\theta(N)_g : G_0(Y') \longrightarrow G_0(X')
\]
defined by 
\[
\theta(N)_g([{\mathcal F}]) = [H_{2 k}(g^*({\Bbb F}.)\otimes_{{\mathcal O}_{Y'}}{\mathcal F})]
- [H_{2 k -1}(g^*({\Bbb F}.)\otimes_{{\mathcal O}_{Y'}}{\mathcal F})] ,
\]
for a sufficiently large $k$,
where ${\mathcal F}$ is a coherent ${\mathcal O}_{Y'}$-module.
It is easy to check that $\theta(N)_g$ is well-defined.
\end{rm}
\end{Definition}

With notation as in Definition~\ref{bivariantK}, we obtain a collection of homomorphisms
\[
\{ 
\theta(N)_g : G_0(Y') \longrightarrow G_0(X') \mid
\mbox{$g : Y' \rightarrow Y$ is a morphism of schemes} 
\} .
\]

\begin{Theorem}
With notation as in Definition~\ref{theta}, the collection of homomorphisms
\[
\{ 
\theta(N)_g : G_0(Y') \longrightarrow G_0(X') \mid
\mbox{$g : Y' \rightarrow Y$ is a morphism of schemes} 
\}
\]
is a bivariant class on K-groups for $X \hookrightarrow Y$.
\end{Theorem}

\proof
Let ${\Bbb F}.$ be an $A$-free resolution of $N$.

First we prove (${\rm B}_1$).
Consider the diagram~(\ref{eq2}).
For a coherent ${\mathcal O}_{Y"}$-module ${\mathcal F}$,
we shall prove 
\[
h'_*\theta(N)_{gh}([{\mathcal F}]) = \theta(N)_{g}h_*([{\mathcal F}]) .
\]
It is enough to show that the both sides are equal to 
\[
[H_{2k}(Rh_*((gh)^*({\Bbb F}.) \otimes_{{\mathcal O}_{Y"}}  {\mathcal F}))]
- [H_{2k - 1}(Rh_*((gh)^*({\Bbb F}.) \otimes_{{\mathcal O}_{Y"}}  {\mathcal F}))]
\]
for a sufficiently large $k$.

Let 
\[
0 \rightarrow {\mathcal F} \rightarrow {\mathcal I} ^{\displaystyle \cdot}
\]
be an injective resolution of ${\mathcal F}$.
Consider the double complex
\[
h_*((gh)^*({\Bbb F}.) \otimes_{{\mathcal O}_{Y"}}  {\mathcal I} ^{\displaystyle \cdot})
= g^*({\Bbb F}.) \otimes_{{\mathcal O}_{Y'}}  h_*({\mathcal I} ^{\displaystyle \cdot}) .
\]
Since $(gh)^*({\Bbb F}.) \otimes_{{\mathcal O}_{Y"}}  {\mathcal I} ^{\displaystyle \cdot}$ is an injective resolution
of $(gh)^*({\Bbb F}.) \otimes_{{\mathcal O}_{Y"}}  {\mathcal F}$,
the $k$th homology of the total complex of the above double complex is
\[
H_k(Rh_*((gh)^*({\Bbb F}.) \otimes_{{\mathcal O}_{Y"}}  {\mathcal F})) .
\]
Consider the spectral sequence
\[
E_2^{p, q} = H_p(g^*({\Bbb F}.) \otimes_{{\mathcal O}_{Y'}}  R^qh_*({\mathcal F}))\Longrightarrow 
H_{p-q}(Rh_*((gh)^*({\Bbb F}.) \otimes_{{\mathcal O}_{Y"}}  {\mathcal F})) .
\]
Then,
\[
\theta(N)_{g}h_*([{\mathcal F}])
= \sum_q(-1)^q [E_2^{2k, q}] - \sum_q(-1)^q [E_2^{2k-1, q}]
\]
for a sufficiently large $k$.
Here, $E_3$-terms of the spectral sequence are the homologies
of the complex
\[
\cdots \longrightarrow E_2^{p+2, q+1} \longrightarrow E_2^{p, q} 
\longrightarrow E_2^{p-2, q-1} \longrightarrow \cdots 
.
\]
If $p$ is big enough, then
the above complex coincides with
\[
\cdots \longrightarrow E_2^{p, q+1} \longrightarrow E_2^{p, q} 
\longrightarrow E_2^{p, q-1} \longrightarrow \cdots 
.
\]
Therefore, we have
\[
\sum_q(-1)^q [E_2^{2k, q}] - \sum_q(-1)^q [E_2^{2k-1, q}]
= \sum_q(-1)^q [E_3^{2k, q}] - \sum_q(-1)^q [E_3^{2k-1, q}]
\]
for a sufficiently large $k$.
Next, 
$E_4$-terms of the spectral sequence are the homologies
of the complex
\[
\cdots \longrightarrow E_3^{p+3, q+2} \longrightarrow E_3^{p, q} 
\longrightarrow E_3^{p-3, q-2} \longrightarrow \cdots 
.
\]
If $p$ is big enough, then
the above complex coincides with
\[
\cdots \longrightarrow E_2^{2k-1, q+4} \longrightarrow E_2^{2k, q+2} \longrightarrow E_2^{2k-1, q} 
\longrightarrow E_2^{2k, q-2} \longrightarrow E_2^{2k-1, q-4} \cdots 
\]
if $p = 2k - 1$, and
\[
\cdots \longrightarrow E_2^{2k, q+4} \longrightarrow E_2^{2k-1, q+2} \longrightarrow E_2^{2k, q} 
\longrightarrow E_2^{2k-1, q-2} 
\longrightarrow E_2^{2k, q-4} \cdots 
\]
if $p = 2k$.
Therefore, we have
\[
\sum_q(-1)^q [E_3^{2k, q}] - \sum_q(-1)^q [E_3^{2k-1, q}]
= \sum_q(-1)^q [E_4^{2k, q}] - \sum_q(-1)^q [E_4^{2k-1, q}]
\]
for a sufficiently large $k$.
Repeating this argument, we obtain
\begin{eqnarray*}
&  & \theta(N)_{g}h_*([{\mathcal F}]) \\
& = & \sum_q(-1)^q [E_2^{2k, q}] - \sum_q(-1)^q [E_2^{2k-1, q}] \\
& = & \sum_q(-1)^q [E_\infty^{2k, q}] - \sum_q(-1)^q [E_\infty^{2k-1, q}] \\
& = & 
[H_{2k}(Rh_*((gh)^*({\Bbb F}.) \otimes_{{\mathcal O}_{Y"}}  {\mathcal F}))]
- [H_{2k - 1}(Rh_*((gh)^*({\Bbb F}.) \otimes_{{\mathcal O}_{Y"}}  {\mathcal F}))] 
\end{eqnarray*}
in $G_0(X')$ for a sufficiently large $k$.

We choose another injective resolution of the complex 
$(gh)^*({\Bbb F}.) \otimes_{{\mathcal O}_{Y"}}  {\mathcal F}$.
We take a double complex ${\Bbb H}..$ of ${\mathcal O}_{Y"}$-modules such that
\begin{itemize}
\item
the total complex of ${\Bbb H}..$ is an injective resolution of
$(gh)^*({\Bbb F}.) \otimes_{{\mathcal O}_{Y"}}  {\mathcal F}$,
\item
$H^q({\Bbb H}..)$ coincide with 
$g^*({\Bbb F}.)\otimes_{{\mathcal O}_{Y'}} R^qh_*({\mathcal F})$
for each $q$,
\item
$H_p({\Bbb H}..)$ is the injective resolution of $H_p((gh)^*({\Bbb F}.) \otimes_{{\mathcal O}_{Y"}}  {\mathcal F})$ for each $p$,
\item
$h_*(H_p({\Bbb H}..)) = H_p(h_*({\Bbb H}..))$ for each $p$.
\end{itemize}
%Both ${\Bbb H}..$ and $(gh)^*({\Bbb F}.) \otimes_{{\mathcal O}_{Y"}}  %{\mathcal I} ^{\displaystyle \cdot}$ 
%are injective resolution of $(gh)^*({\Bbb F}.) \otimes_{{\mathcal O}_{Y"}}  {\mathcal F}$.
Consider the following spectral sequence.
\[
{}'E_2^{p, q} = H^qH_p(h_*({\Bbb H}..)) =
 R^qh'_*(H_p((gh)^*({\Bbb F}.) \otimes_{{\mathcal O}_{Y"}}  {\mathcal F}))
\Longrightarrow 
H_{p-q}(Rh_*((gh)^*({\Bbb F}.) \otimes_{{\mathcal O}_{Y"}}  {\mathcal F})) .
\]
Then,
\[
h'_*\theta(N)_{gh}([{\mathcal F}])
= \sum_q(-1)^q [{}'E_2^{2k, q}] - \sum_q(-1)^q [{}'E_2^{2k-1, q}]
\]
for a sufficiently large $k$.
Here, we have
\begin{eqnarray*}
& & \sum_q(-1)^q [{}'E_2^{2k, q}] - \sum_q(-1)^q [{}'E_2^{2k-1, q}] \\
& = & \sum_q(-1)^q [{}'E_\infty^{2k, q}] - \sum_q(-1)^q [{}'E_\infty^{2k-1, q}] \\
& = & 
[H_{2k}(Rh_*((gh)^*({\Bbb F}.) \otimes_{{\mathcal O}_{Y"}}  {\mathcal F}))]
- [H_{2k - 1}(Rh_*((gh)^*({\Bbb F}.) \otimes_{{\mathcal O}_{Y"}}  {\mathcal F}))] 
\end{eqnarray*}
in $G_0(X')$ for a sufficiently large $k$.
It will be proved in the same way as in the case $E_*^{p,q}$ as above. 
The condition (${\rm B}_1$) is proved.

It is easy to prove (${\rm B}_2$).

Now we  prove (${\rm B}_3$).
Consider the diagram~(\ref{B3}).
For a coherent ${\mathcal O}_{Y'}$-module ${\mathcal F}$,
we shall prove 
\[
\chi_{(hf)^*({\Bbb G}.)} \theta(N)_{g}([{\mathcal F}]) = 
\theta(N)_{gi'}\chi_{h^*({\Bbb G}.)}([{\mathcal F}]) .
\]
It is enough to show that the both sides are equal to 
\[
[H_{2k}(g^*({\Bbb F}.) \otimes_{{\mathcal O}_{Y'}} h^*({\Bbb G}.) \otimes_{{\mathcal O}_{Y'}} {\mathcal F})]
- [H_{2k-1}(g^*({\Bbb F}.) \otimes_{{\mathcal O}_{Y'}} h^*({\Bbb G}.) \otimes_{{\mathcal O}_{Y'}} {\mathcal F})]
\]
for a sufficiently large $k$.
The proof will be done in the same way as in (${\rm B}_1$), so we omit it.
\qed

\section{Theta invariant for hypersurface with isolated singularity}\label{mainresult}

Let $(A,\m)$ be a hypersurface with isolated singularity.
Then, the theta invariant induces a map
\[
\theta^A : G_0(A)\otimes G_0(A) \longrightarrow {\Bbb Z}
\]
by $\theta^A(\sum_i\pm [M_i]\otimes[N_i]) = \sum_i\pm \theta^A(M_i, N_i)$.
We sometimes denote $\theta^A(\alpha\otimes \beta)$ simply by $\theta^A(\alpha, \beta)$.

By the following theorem, we know that
$\theta^A(M,N) = 0$ if either $M$ or $N$ is numerically equivalent to $0$.
See Corollary~\ref{cornum}.

\begin{Theorem}\label{maintheorem}
Let $(A, \m)$ be a hypersurface with isolated singularity.
Assume $\dim A > 0$.
Let $N$ be a finitely generated $A$-module.

We assume that there exists a resolution of singularity of $\spec A$,
i.e., a proper birational morphism $\pi : Z \rightarrow \spec A$
such that $Z$ is regular.
Put $W = \pi^{-1}(m)$ and $U = Z \setminus W$.
Assume that $U$ is isomorphic to $\spec A \setminus \spec A/m$.

Then, there exist $A$-modules $N_1$ and $N_2$ such that
\begin{itemize}
\item
$\ell_A(N_1) = \ell_A(N_2) < \infty$
\item
${\rm pd}_A(N_1) = {\rm pd}_A(N_2) = \dim A$
\item
$\theta^A(N, M) = \chi^A(N_1,M) - \chi^A(N_2,M)$ for any finitely generated $A$-module $M$.
\end{itemize}
\end{Theorem}

\proof
First, we shall prove
\[
i_*\theta(N)_\pi ([{\mathcal O}_Z]) = 0
\]
in $G_0(Z)$, where $i : W \rightarrow Z$ be the closed immersion.
Consider the minimal $A$-free resolution ${\Bbb F}.$ of $N$.
Then, it is written as
\[
\cdots \stackrel{\alpha}{\longrightarrow}
F_{2k+1} \stackrel{\beta}{\longrightarrow}
F_{2k} \stackrel{\alpha}{\longrightarrow}
F_{2k-1} \stackrel{\beta}{\longrightarrow} \cdots
\]
for $k \gg 0$.
Here, $F_n$'s are of the same rank for $n \gg 0$.
Here, consider the complex $\pi^*({\Bbb F}.)$
\[
\cdots \stackrel{\pi^*\alpha}{\longrightarrow}
\pi^*F_{2k+1} \stackrel{\pi^*\beta}{\longrightarrow}
\pi^*F_{2k} \stackrel{\pi^*\alpha}{\longrightarrow}
\pi^*F_{2k-1} \stackrel{\pi^*\beta}{\longrightarrow} \cdots .
\]
Then,
\[
\theta(N)_\pi ([{\mathcal O}_Z])
= [H_{2k}(\pi^*({\Bbb F}.))] - [H_{2k-1}(\pi^*({\Bbb F}.))] \in G_0(W)
\]
for $k \gg 0$.
Let $K_\alpha$ (resp.\ $I_\alpha$) be the kernel (resp.\ image)
of the map $\pi^*\alpha$.
Let $K_\beta$ (resp.\ $I_\beta$) be the kernel (resp.\ image)
of the map $\pi^*\beta$.
Then, we have the following exact sequences
\[
\begin{array}{l}
0 \longrightarrow K_\alpha \longrightarrow \pi^*F_{2k}
\longrightarrow I_\alpha \longrightarrow 0 \\
0 \longrightarrow K_\beta \longrightarrow \pi^*F_{2k-1}
\longrightarrow I_\beta \longrightarrow 0 \\
0 \longrightarrow I_\beta \longrightarrow K_\alpha \longrightarrow 
H_{2k}(\pi^*({\Bbb F}.)) \longrightarrow 0 \\
0 \longrightarrow I_\alpha \longrightarrow K_\beta \longrightarrow 
H_{2k-1}(\pi^*({\Bbb F}.)) \longrightarrow 0 .
\end{array}
\]
Therefore,
we obtain
\begin{eqnarray*}
i_*\theta(N)_\pi ([{\mathcal O}_Z]) & = &  
[H_{2k}(\pi^*({\Bbb F}.))] - [H_{2k-1}(\pi^*({\Bbb F}.))] \\
& = & ([K_\alpha] - [I_\beta]) - ([K_\beta] - [I_\alpha]) \\
& = & [\pi^*F_{2k}] - [\pi^*F_{2k-1}] = 0
\end{eqnarray*}
in $G_0(Z)$. 

Remember that,
since $A$ is Cohen-Macaulay,
the Grothendieck group of bounded $A$-free complexes
with support in $\{ m \}$ is generated by finite free resolutions of modules of finite length and finite projective dimension (see Proposition~2 in \cite{RS}).
Therefore, by Theorem~\ref{maintheorem}, there exist $A$-modules $N_1$ and $N_2$ of finite length and of finite projective dimension such that, letting ${\Bbb A}.$ (resp.\ ${\Bbb B}.$) be a finite $A$-free resolution
of $N_1$ (resp.\ $N_2$), 
we have 
\[
\theta(N)_g \otimes 1= 
\chi({\Bbb A}.\oplus {\Bbb B}.(-1))_g \otimes 1= 
\chi({\Bbb A}.)_g \otimes 1 - \chi({\Bbb B}.)_g \otimes 1 
\]
for any $g : Y' \rightarrow \spec A$.

In particular, we have
\[
\chi({\Bbb A}.)_{{\rm id}_{\spec A}} - \chi({\Bbb B}.)_{{\rm id}_{\spec A}} =
\chi({\Bbb A}.\oplus {\Bbb B}.(-1))_{{\rm id}_{\spec A}} =
\theta(N)_{{\rm id}_{\spec A}}
\]
since $G_0(\spec A/m) \simeq {\Bbb Z}$.
Hence, for any finitely generated $A$-module $M$,
we have
\begin{eqnarray*}
\theta^A(N, M) & = & \theta(N)_{{\rm id}_{\spec A}}([M]) \\
& = & \chi({\Bbb A}.)_{{\rm id}_{\spec A}}([M]) - \chi({\Bbb B}.)_{{\rm id}_{\spec A}}([M]) \\
& = & \chi_{{\Bbb A}.} ([M]) - \chi_{{\Bbb B}.}([M]) \\
& = & \chi^A(N_1, M) - \chi^A(N_2, M) .
\end{eqnarray*}
Here, note that 
\[
0 = \theta^A(N,A) = \chi^A(N_1, A) - \chi^A(N_2, A)
= \ell(N_1) - \ell(N_2) .
\]
\qed

In \cite{CW} and \cite{MPSW}, it is  proved that 
the theta invariant vanishes on cycles which are algebraic equivalent 
to $0$.  
The following corollary covers the above fact. Note that algebraic equivalence implies numerical equivalence since Euler characteristic is constant in a flat family. 
It will be discussed in Remark~\ref{algeq} in detail.

The following Corollary~\ref{cornum} contains Theorem~\ref{kazu1}.

\begin{Corollary}\label{cornum}
Let $(A, m)$ be a $d$-dimensional hypersurface with isolated singularity.
Assume $d > 0$.
Assume that there exists a resolution of singularity $\pi : Z \rightarrow \spec A$
such that $Z \setminus \pi^{-1}(m)$ is isomorphic to $\spec A \setminus \spec A/m$.
\begin{enumerate}
\item
Suppose that $\alpha, \beta \in G_0(A)$.
If one of $\alpha$ and $\beta$ is numerically equivalent to $0$ in the sense of \cite{K23}, 
then $\theta^A(\alpha,\beta) = 0$.
\item
If $\overline{G_0(A)}\subq = {\Bbb Q}[A]$ (or equivalently $\overline{A_i(A)}\subq = 0$ for $i < d$), then $\theta^A(M,N) = 0$ for any finitely generated $A$-modules $M$ and $N$.
\end{enumerate}
\end{Corollary}

\proof
First, we shall prove (1).
We assume that $\alpha$ is numerically equivalent to $0$ in the sense of \cite{K23}.
We shall prove that $\theta^A(N,\alpha) = 0$ 
for any finitely generated $A$-module $N$.
By Theorem~\ref{maintheorem}, there exist bounded finite $A$-free complexes 
${\Bbb A}.$ and ${\Bbb B}.$ with homology of finite length such that
\begin{eqnarray*}
\theta^A(N, \alpha) & = & \theta(N)_{{\rm id}_{\spec A}}(\alpha) \\
& = & \chi({\Bbb A}.)_{{\rm id}_{\spec A}}(\alpha) - \chi({\Bbb B}.)_{{\rm id}_{\spec A}}(\alpha) \\
& = & \chi_{{\Bbb A}.} (\alpha) - \chi_{{\Bbb B}.}(\alpha) .
\end{eqnarray*}
It is equal to $0$ since $\alpha$ is numerically equivalent to $0$.

Next, we shall prove (2).
Remember that, for a Noetherian local domain $(A,m)$, 
$\overline{G_0(A)\subq} = {\Bbb Q}[A]$ if and only if $\chi_{{\Bbb G}.}([T]) = 0$
for any finitely generated $A$-module $T$ with $\dim T < \dim A$ and
any bounded finite free $A$-complex ${\Bbb G}.$ with homology finite length.
For a given $A$-module $M$, letting $r$ be the rank of $M$,
we have an exact sequence of the form
\[
0 \longrightarrow A^r \longrightarrow M \longrightarrow T \longrightarrow 0,
\]
where $T$ is an $A$-module with $\dim T < \dim A$.
Then, we have
\[
\theta^A(N,M) = \theta^A(N,A^r) + \theta^A(N,T) = 
\chi_{{\Bbb A}.} ([T]) - \chi_{{\Bbb B}.}([T]) = 0.
\]
For Chow groups of local rings divided by numerical equivalence,
see the following remark.
\qed

\begin{Remark}
For a Notherian local ring $(A,\m)$, $\overline{A_i(A)}\subq = 0$ for 
$i = 0, 1, \ldots, s$ if and only if $\chi_{{\Bbb G}.}([M]) = 0$ 
for any finitely generated $A$-module $M$ with $\dim M \le s$
and any bounded $A$-free complex ${\Bbb G}.$ with support in $\{ m  \}$, see \cite{K23} for details. 
\end{Remark}

\begin{Remark}\label{algeq}
\begin{rm}
For cycles of a local ring $A$,
algebraic equivalence in Definition~3.2 in \cite{CW} implies
numerical equivalence as follows.

Let $A$ be a local ring which is essentially of finite type
over an algebraically closed field $k$.
Let $B$ be an integral domain that is smooth of finite type
over $k$. 
Let $m_{\ell}$ be a maximal ideal of $B$ and
\[
i_\ell : \spec B/m_\ell \longrightarrow \spec B
\]
be the closed immersion for $\ell = 1, 2$.
We define
\[
i_\ell^* : G_0(A\otimes_kB) \longrightarrow G_0(A)
\]
by
\[
i_\ell^*([M]) = \sum_{j}(-1)^j[{\rm Tor}^B_j(M, B/m_\ell)] .
\]
Cycles of the form
\[
i_1^*(\gamma) - i_2(\gamma)
\]
generates algebraic equivalence, where $\gamma \in  G_0(A\otimes_kB)$.

We shall prove that the above cycle is numerically equivalent to $0$.
Let ${\Bbb F}.$ be a bounded finite $A$-free complex with homology finite length.
We have
\[
\chi_{{\Bbb F}.}(i_1^*(\gamma) - i_2(\gamma))
= i_1^*\chi_{{\Bbb F}.\otimes_kB}(\gamma) - 
i_2\chi_{{\Bbb F}.\otimes_kB}(\gamma)
\]
where $i_\ell^* : G_0(B) \longrightarrow G_0(k)$ in the 
right-hand side is defined by 
\[
i_\ell^*([L]) = \sum_j(-1)^j[{\rm Tor}^B_j(L, B/m_\ell)] .
\]
Using the vanishing theorem of intersection multiplicities,
we obtain
\[
i_\ell^*\chi_{{\Bbb F}.\otimes_kB}(\gamma)
= \rank_A \chi_{{\Bbb F}.\otimes_kB}(\gamma)
\]
for $\ell = 1, 2$.
Thus, $i_1^*\chi_{{\Bbb F}.\otimes_kB}(\gamma) = 
i_2\chi_{{\Bbb F}.\otimes_kB}(\gamma)$.
\end{rm}
\end{Remark}

%It is equivalent to that ${\rm ch}({\Bbb G}.)(c) = 0$
%for any $c \in A_i(A)$ with $i < \dim A$ and
%any bounded finite free $A$-complex ${\Bbb G}.$ with homology finite length,
%where ${\rm ch}({\Bbb G}.)$ the localized Chern character (18.1 in \cite{F}) 
%of  the complex ${\Bbb G}.$. 

\section{Various cases of Conjecture \ref{modconj} and other applications} \label{psd}
In this Section we prove various cases of Conjecture \ref{modconj}. Some of these results are already known but we provide alternating proof. Some, such as the positive semi-definite of $\theta^A$ in dimension $3$ (Theorem \ref{UFD}) is new. As some applications, we extend the example of Dutta-Hochster-McLaughlin to all 
non-factorial isolated hypersurface singularities of dimension three and compute the Grothendieck group modulo numerical equivalence for such rings (Corollary \ref{dhm} and \ref{last}). 

\begin{Remark}
\begin{rm}
By Theorem~\ref{maintheorem} and Proposition~\ref{affirm}, Conjecture~\ref{modconj}, part (1), (2) and (3)  are true for quasi-homogeneous case. This recovers a result of Moore-Piepmeyer-Spiroff-Walker~\cite{MPSW}.
We just need to note that the blow-up of $\spec A$ at the maximal ideal gives a resolution of singularity in this case.

Let $(A, m)$ be a $d$-dimensional hypersurface with isolated singularity.
Assume that there exists a resolution of a singularity $\pi : Z \rightarrow \spec A$
such that $Z \setminus \pi^{-1}(m)$ is isomorphic to $\spec A \setminus \spec A/m$.
Assume that $d$ is even.
If Conjecture~\ref{conj} (1) (a) is true, 
then $\theta(M,N) = 0$ by Theorem~\ref{maintheorem} for any finitely generated $A$-modules $M$ and $N$,
that is, Conjecture~\ref{modconj}(1) is true.
We remark that Conjecture \ref{modconj}(1) is true in the case of characteristic zero 
by Buchweitz-Straten \cite{BV}.
\end{rm}
\end{Remark}

\begin{lem}\label{sumext}
Let $A$ be a local ring and $N$ be a finitely generated module and $l\geq 0$ be an integer. Then in $\underline {G_0}(A)$
$$ 
(-1)^l[\syz^l(N)^*] = \sum_{i=0}^l (-1)^i[\Ext_A^i(N,A)] .
$$
\end{lem}

\begin{proof}
This is easy  by induction on $l$.
\end{proof}

For an element $\alpha \in \underline{G_0}(A)$, let $\dim(\alpha)$ be the smallest integer $c$ such that $\alpha$ has a representation by formal sum (with coefficients)  of modules of dimensions at most $c$.

\begin{lem} \label{dual}
Suppose $A$ is Gorenstein of dimension $d$ and $N$ is of codimension $a$ and codepth $b$. Let $M = \syz^b(N)$. Then:
$$\dim ((-1)^b[M^*] - (-1)^a[N]) < d-a$$ 
\end{lem}

\begin{proof}
For each minimal prime $p$ of $N$, we have
$$\ell_{A_p}(N_p)  = \ell_{A_p} (\Ext_A^a(N,A)_p)$$ 
which shows $\dim([N] - [\Ext_A^a(N,A)])<d-a$.  Lemma \ref{sumext} applies to get the desired conclusion, since the support of 
$\Ext_A^i(N,A)$ evidently has dimension less than $d-a$
for $i \neq a$. 
\end{proof}

\begin{prop}
Let $\dim A=6$. Then the statements (\ref{con1}), (\ref{con3}) of Conjecture \ref{modconj} are equivalent. 
\end{prop}

\begin{proof}
Clearly (\ref{con1}) implies (\ref{con3}). Clearly,  it is enough to show (\ref{con1}) for MCM modules. Let $M$ be such a module. Then take a Bourbaki sequence:
$$ \ses{F}{M}{I}$$

Since $d=6$, we know that $A$ is a UFD. It follows that $\height I \geq 2$ and $\dim A/I \leq 4$. But counting depths shows that $\depth I\geq 5$, and thus $\depth A/I  = \dim A/I =4$. Hence it will be enough to prove $\h{A}{A/I}{A/J} = 0$ for any height two ideals $I,J$ such that the quotients are Cohen-Macaulay.   
 
Now we let $M = \syz^2(A/I)$ and $N = \syz^2(A/J)$. By \cite[4.2, 4.3]{Da2} we know:
$$\h{A}{M}{N} + \h{A}{M^*}{N^*} =0$$

By Lemma \ref{dual} we know that  $[M^*] =  [A/I] +\alpha$ and $[N^*] = [A/J] +\beta$ in $\underline{G_0}(A)$, with $\dim \alpha$ and $\dim \beta$ at most three. Evidently, we also know $[M] = [A/I]$ and $[N] = [A/J]$ in $\underline{G_0}(A)$.  Since $\theta^A$ vanishes when one of the argument has dimension at most three, it follows then that:
$$
\h{A}{A/I}{A/J}  + \h{A}{A/I}{A/J}   =0
$$    
as we wanted.
\end{proof}

\begin{cor}
Let  $A$ be an excellent local hypersurface with isolated singularity
containing a field of characteristic $0$ and suppose $\dim A\leq 6$. Then Conjecture \ref{modconj}(3) holds. 
\end{cor}

Conjecture \ref{modconj}(5) is known in the standard graded case over a field of  characteristic $0$, see \cite{MPSW}. We shall establish it for $d = \dim A \leq 3$.

\begin{prop}
If $d=1$, Conjecture   \ref{modconj}(5) is true. 
\end{prop}

\begin{proof}
Let $p_1,p_2,...,p_n$ be the minimal primes of $A$.
Then $G_0(A)_{\mathbb{Q}}$ has a basis consisting of 
the classes $[A/p_1],...,[A/p_n]$. In
particular, since $A$ has dimension 1 and is reduced, $[A] =
\sum_{1}^{n}[A/p_i]$. Let $\alpha_{ij} = \theta^A(A/p_i,A/p_j)$.
For $i \neq j$, $p_i + p_j$ is $m$-primary, and it is not hard
(using the resolution of $A/p_i$, noting that each $p_i$ is a principal ideal) to see that $\alpha_{ij} = \ell(A/(p_i+p_j)) >0$.
Since $\theta^A(A,A/p_i) = 0$, we must have $\alpha_{ii} =
-\sum_{j\neq i} \alpha_{ij}$. Now, for a  module $M$,  $[M] =
\sum a_i[A/p_i]$, here $a_i$ is the rank of $M_{p_i}$. Then 
$$ 
\h{A}{M}{M} = \sum_{i,j} \alpha_{ij}a_ia_j = -\sum_{i<j} \alpha_{ij}(a_i-a_j)^2 \leq 0 .
$$
Clearly, the equality happens iff   $a_1 = a_2 = ... =a_n$. But  then $[M]=a_1[A]$, so $\theta^A(M,M) =
a_1\theta^A(A,M)= 0$. 
\end{proof}

Next, we look at dimension $3$. The proof here follows and improves the main result of \cite{Da4}. The divisor class map as described in Section \ref{prep} will play an important role. We start with a more general result:

\begin{thm}\label{mainTheorem}
Let $A$ be (abstract) local hypersurface of dimension $3$. Let $M,N$ be reflexive $A$-modules which are  locally free  on $U_A= \spec A -\{m\}$, the punctured spectrum of $A$. Suppose $\Hom_A(M,N)$ is a maximal Cohen-Macaulay $A$-module. Then $\h{A}{M^*}{N}\leq 0$. Furthermore, equality happens if and only if $M$ or $N$ is free. 
\end{thm}

\begin{proof}
First, there is no loss of generality by completing everything in sight, so we will assume $A$ is complete. So $A\cong S/(f)$, where $S$ is a regular local ring of dimension $4$.

Suppose that $\Hom_A(M,N)$ is maximal Cohen-Macaulay. Then by \cite[Lemma 2.3]{Da3} we have $\Ext^1_A(M,N)=0$. 
One has the following short exact sequence (see \cite[3.6]{Ha} or \cite{Jo}, \cite{Jor}):
$$ \Tor_2^A(M_1,N) \to \Ext^1_A(M,A)\tensor_A N \to \Ext^1_A(M,N) \to \Tor_1^A(M_1,N) \to 0 $$
Here $M_1$ is the cokernel of $F_1^* \to F_2^*$, where $\cdots \to F_2 \to F_1\to F_0 \to M \to 0$ is the minimal free resolution of $M$.  Since $\Ext^1_A(M,N)=0$, it follows that $\Tor_1^A(M_1,N) =0$.

The change of rings long exact sequence for $\Tor$ (\cite{HW1}) yields a surjection 
\[
\Tor_3^S(M_1,N) \twoheadrightarrow \Tor^A_3(M_1,N) .
\]
 We claim that $\Tor_i^S(M_1,N)=0$ for $i\geq 3$.  As $N$ is reflexive we have $\depth N\geq 2$. It follows that as an $S$-module, the depth of $N$ is also at least $2$, so $\pd_SN\leq 4-2=2$ and our assertion follows.    

It follows that $\Tor^A_3(M_1,N)=0$. Also, $\Tor_i^A(M_1,N)$ becomes periodic of period $2$ after $i\geq 2$. So $\h {A}{M_1}{N} = \ell(\Tor_2^A(M_1,N))\geq 0$. 

Finally, we have by definition a long  sequence: $$ 0 \to M^* \to F_0^* \to F_1^* \to F_2^* \to M_1 \to 0 $$  
By the acyclic lemma, we know $\Ext^1_A(M,A)=0$. So in  
$\underline{G_0}(A)_{\mathbb Q}$ we have $[M^*]=-[M_1]$, and as $N$ is locally free  on $U_A$ it follows that $\h{A}{M^*}{N}= - \h{A}{M_1}{N}\leq 0$. 

Now we prove the last claim. Looking at the proof, we see that equality happens if and only if $\Tor_2^A(M_1,N) =  0$. As $\Tor_1^A(M_1,N) =0$, it now follows that $\Tor_i^A(M_1,N)=0$ for all $i>0$. 
Note that the long exact sequence at the beginning of the proof now implies $\Ext^1_A(M,A)=0$, thus $M$ is MCM. It also follows that $M^*$ is third syzygy of $M_1$, so  $\Tor_i^A(M^*,N)=0$ for all $i>0$. By the depth formula (Proposition~2.5 in \cite{HW1}) one has
$$\depth M^* +\depth N =\depth A + \depth M^*\otimes_A N $$
Thus $\depth M^*\otimes_A N =\depth N \geq 2$.  But the canonical map $M^*\otimes_A N \to \Hom_A(M,N)$ has finite length kernel and cokernel, so it must be an isomorphism. We conclude that $\depth N=3$, i.e., $N$ is also MCM.

As shown above, $\Tor_i^A(M^*,N)=0$ for all $i>0$, so either $M^*$ or $N$ has finite projective dimension by \cite[Theorem 1.9]{HW2} or \cite[1.1]{Mil}. But since they are both MCM, one of them must be free.
\end{proof}

\begin{lem}\label{mainLem}
Let $A$ be a local hypersurface of dimension $3$ with isolated singularity and $M,N$ be reflexive $A$-modules. Let $[I],[J]\in \Cl(A)$ represent $c_1([M]),c_1([N])$ respectively. We have $\h{A}{M}{N}=\h{A}{I}{J}$. Furthermore $\h{A}{M}{N}= - \h{A}{M}{N^*}$.
\end{lem}

\begin{proof}
It is not hard to see that in $\underline {G_0}(A)_{\mathbb Q}$, the reduced Grothendieck group with rational coefficients, we have an equality $[N] -[J] = \sum a_i[A/P_i]$ such that each  $P_i \in \spec A$ has height at least $2$ (see the proof of 3.1 in \cite{Da4} ).     

The first half will be proved by showing that $\h AM{A/P}=0$ for each $P \in \spec A$ such that $\height P= 2$. 

By  \cite[Theorem 1.4]{HJW} one can  construct a Bourbaki sequence for $M$:
$$\ses FMI $$
such that $I \not\subseteq P$.  Obviously $\h AM{A/P} =\h AI{A/P}$.  But $A/I\tensor_AA/P$ has finite length, and $\dim A/I +\dim A/P \leq 3 =\dim A $. By \cite{Ho1},  $\h A{A/I}{A/P}=0$. Since $\h AI{A/P}=-\h A{A/I}{A/P}$, we have $\h AM{A/P}=0$. 

The last statement follows from $[N^*] - [J^*] = 
\sum b_i[A/P_i]$ in $\underline{G_0}(A)$, where each $P_i \in \spec A$ has height at least $2$.
\end{proof}

The following contains Corollary~\ref{dim3} in the Introduction. 

\begin{thm}\label{UFD}
Let $A$ be a local  hypersurface with isolated singularity and $\dim A=3$. Let $M$ be a finitely generated $A$-module. Then $ \h{A}{M}{M} \geq 0$.   The equality happens if and only if $c_1([M])= 0$ in  $\Cl(A)$. 
\end{thm}

\begin{proof}
By taking high syzygy we can assume that $M$ is maximal Cohen-Macaulay. 
Let $I$ represent $c_1([M])\in \Cl(A)$. 
Then $I^*$ represents $c_1([M^*])$. 
Since $\Hom_A(I,I)\cong A$ we know by Theorem \ref{mainTheorem} and Lemma \ref{mainLem} that $\h{A}{M}{M} \geq 0$, with equality if and only if $[I]$ is $0$ in $\underline {G_0}(A)_{\mathbb Q}$. 
This condition imply $c_1([M])=0$.  
\end{proof}

The following contains Corollary~\ref{DHM} in the Introduction. 

\begin{cor}\label{dhm}
Let $A$ be a local  hypersurface with isolated singularity and $\dim A=3$. Assume that the assumptions of Theorem \ref{maintheorem} holds. Let $M$ be a finitely generated, torsion $A$-module. The following are equivalent:

\begin{enumerate}
\item $[M]$ is numerically equivalent to $0$ in $G_0(A)$. 
\item $\theta^A(M,M) = 0$. 
\item $c_1([M])=0$ in $\Cl(A)$.
\end{enumerate}
\end{cor}

\begin{proof}
The result follows from  Corollary~\ref{cornum}, Theorem~\ref{UFD} and \cite[Proposition 3.7]{K23}. 
\end{proof}

Finally it is worth pointing out that one can now compute the Grothendieck group modulo numerical equivalence for most isolated hypersurface singularities in dimension three.

\begin{cor}\label{last}
Let $A$ be a local  hypersurface with isolated singularity and $\dim A=3$. Assume that the assumptions of Theorem \ref{maintheorem} holds. Then the following holds:
\begin{enumerate}
\item $\overline {G_0(A)} = \mathbb Z[A] \oplus A_2(A)$.
\item The class group of $A$ is finitely generated and torsion free. 
\end{enumerate}
\end{cor}

\begin{proof}
Statement (1) follows from Corollary \ref{dhm}. 
Since $\overline{G_0(A)}$ is finitely generated torsion-free
by Theorem~3.1 and Remark~3.5 in \cite{K23}, so is  $A_2(A)$. 
\end{proof}

\section*{Acknowledgments}
The proof of Theorem \ref{UFD} was essentially worked out  with Mark Walker during a visit of the first author to University of Nebraska in March 2010. The first author would like to thank University of Nebraska for its hospitality and Mark Walker for allowing him to include the result here. The second author would like to thank University of Kansas for its hospitality during his visits in May 2011 and July 2012, during which most of this work was done.


\begin{thebibliography}{99}









\bibitem{Ba} L. Badescu, \textit{A remark on the Grothendieck-Lefschetz theorem about the Picard group}, Nagoya Math.\ J.\
{71} (1978), 169--179.


\bibitem{BV} R-.O Buchweitz, D. van Straten, \emph{An Index Theorem for Modules on a Hypersurface Singularity}, preprint, arXiv:1103.5574.

\bibitem{CW} O. Celikbas, M. Walker, \emph{Hochster's theta pairing and Algebraic equivalence},  Math. Annalen, to appear. 

\bibitem{Ch} C.-Y. J. Chan, \emph{Filtrations of modules, the Chow group, and the Grothendieck
group}, J. Algebra {219} (1999), 330--344.




\bibitem{Da2} H. Dao, \emph{Some observations on local and projective hypersurfaces}, Math. Res. Let. 15 (2008), no. 2, 207--219. 

\bibitem{Da3} H. Dao, \emph{Remarks on non-commutative crepant resolutions}, Advances in Math., 224 (2010), 1021--1030.

\bibitem{Da4}H. Dao, \emph{Picard groups of punctured spectra of dimension three local hypersurfaces are torsion-free}, Compositio Math, 148 (2012), 145--152. 

\bibitem{Da1} H. Dao, \emph{Decency and rigidity for modules over local  hypersurfaces},  Transactions of the AMS, to appear. 




\bibitem{De} P.\ Deligne,  \textit{Cohomologie des intersections completes, Sem.
Geom. Alg. du Bois Marie (SGA 7, II)}, Springer Lect. Notes Math.,
No. 340, (1973).

%\bibitem{DLM} H. Dao, J. Li, C. Miller, \emph{On (non)rigidity of the Frobenius over Gorenstein rings}, Algebra and Number Theory {4-8} (2010), 1039--1053.


\bibitem{Ei} D. Eisenbud, \emph{Homological algebra on a complete intersection, with an application to group representations}, Tran. Amer. Math. Soc. {260}
(1980), 35--64.

%\bibitem{FOV} H. Flenner, L. O'Carroll, W. Vogel, \emph{Joins and intersections}, Springer Mono. Math., (1999). 

\bibitem{F} { W.~Fulton},
{\it Intersection Theory, 2nd Edition},
Springer-Verlag, Berlin, New York, 1997.

\bibitem{Gab} O. Gabber, \emph{On purity for the Brauer group}, Arithmetic Algebraic Geometry, Oberwolfach Report No. 34 (2004), 1975--1977.   


%\bibitem{HW1} C. Huneke, R. Wiegand, \emph{Tensor products of modules and the rigidity of Tor},
%Math. Ann. 299 (1994), 449--476.

\bibitem{Ha} R. Hartshorne, \emph{Coherent functors}, Adv. in Math. 140 (1994), 44--94.


\bibitem{Ho1} M. Hochster, \emph{The dimension of an intersection in an ambient
hypersurface}, Proceedings of the First Midwest Algebraic Geometry Seminar (Chicago Circle,1980),
Lecture Notes in Mathematics {862}, Springer-Verlag, 1981, 93--106.

\bibitem{Hor} G. Horrocks, \emph{Vector bundles on the punctured spectrum of a local ring}, Proc. London Math. Soc. 14 (1964), 689--713.

\bibitem{HW1} C. Huneke, R. Wiegand, \emph{Tensor products of modules and the rigidity of Tor},
Math. Ann. 299 (1994), 449--476.

\bibitem{HW2} C. Huneke, R. Wiegand, \emph{Tensor products of modules, rigidity and local cohomology},
Math. Scan. {81} (1997), 161-183.



\bibitem{HJW} C. Huneke, R. Wiegand, D. Jorgensen,  \emph{Vanishing theorems for complete
intersections}, J. Algebra {238} (2001), 684--702.

\bibitem{Jor} D. Jorgensen, \emph{Finite projective dimension and the vanishing of Ext(M,M)}, Comm. Alg. 36 (2008) no. 12, 4461--4471.

\bibitem{Jo} P. Jothilingam, \emph{A note on grade}, Nagoya Math. J. 59 (1975), 149–-152


%\bibitem{K} H. Kn\"{o}rrer, \emph{Cohen-Macaulay modules on hypersurface singularities. I}, Invent. Math., 
%88 (1987), 153–-164.

\bibitem{K23} {K.~Kurano},
{\it Numerical equivalence defined on Chow groups of Noetherian local rings},
Invent. Math., {\bf 157} (2004), 575--619.


\bibitem{Lich} S. Lichtenbaum, \emph{On the vanishing of Tor in regular local rings}, Illinois J. Math. 10 (1966), 
220–-226. 

%\bibitem{Lin} H.W. Lin, \emph{ On crepant resolution of some hypersurface singularities and a criterion for 
%UFD}, Trans. Amer. Math. Soc., 354 (2002), no.5, 1861–-1868.


\bibitem{mat}
H. Matsumura, \emph{Commutative Ring Theory}, Cambridge
Studies in Advanced Mathematics {8}, Cambridge University
Press, Cambridge (1986).

\bibitem{Mil} C. Miller, \emph{Complexity of Tensor Products of Modules and a Theorem of Huneke-Wiegand}, Proc. Amer. Math. Soc. 126 (1998), 53--60.

\bibitem{MPSW} F. Moore, G. Piepmeyer, S. Spiroff, M. Walker, \emph{Hochster's theta invariant and the Hodge-Riemann bilinear relations}, Advances in Math, 226 (2010), no. 2, 1692--1714.

\bibitem{PV}
A. Polishchuk, A. Vaintrob, \emph{Chern characters and Hirzebruch-Riemann-Roch formula for matrix factorizations}, preprint (2010). 

\bibitem{Rob} L. Robbiano, \emph{Some properties of complete intersections in good projective varieties},
Nagoya Math. J., 61 (1976), 103--111. 


\bibitem{Ro} P. Roberts, \emph{Multiplicities and Chern classes in Local Algebra},
Cambridge Univ. Press, Cambridge (1998).


\bibitem{RS} {P. C. Roberts and V. Srinivas},
{\it Modules of finite length and finite projective dimension},
Invent.\ Math., {\bf 151} (2003), 1--27.


\bibitem{TT} {R. W. Thomason and T. Trobaugh},
{\it Higher algebraic $K$-theory of schemes and of derived categories}. 
``The Grothendieck Festschrift", 
Vol. \RMN{3}, 247--435, 
Progr. Math., 88, 
Birkhauser Boston, Boston, MA, 1990.

\bibitem{V1} M. Van den Bergh, \emph{Non-commutative crepant resolutions}, The legacy of Niels Henrik Abel, 
 Springer, Berlin, (2004) 749–-770. 



\end{thebibliography}
\end{document}